\documentclass[12pt,reqno]{amsart}
\usepackage{amssymb,amsmath,amsfonts,latexsym}
\usepackage{bm}
\usepackage{mathtools}
\usepackage{enumerate}
\usepackage{hyperref}

\usepackage{array,graphics,color}
\numberwithin{equation}{section}

\newtheorem{dfn}{Definition}[section]
\newtheorem{thm}[dfn]{Theorem}
\newtheorem{lma}[dfn]{Lemma}

\newtheorem{rmrk}[dfn]{Remark}

\DeclarePairedDelimiterX{\norm}[1]{\lVert}{\rVert}{#1}

\begin{document}
	
	\title[Geodesic Bounds of  Hyperbolic Link Complements in Hyperbolic $3$-Manifolds]{Closed Geodesic Length Bounds of  Hyperbolic Link Complements in Hyperbolic $3$-Manifolds}
	
	
	\author[Ghosh] {Buddha Dev Ghosh}
	\address{Department of Mathematics, Indian Institute of Technology Guwahati, Guwahati, 781039, India}
	\email{ghosh176123006@iitg.ac.in, budo2012dec@gmail.com}

	
	
	\begin{abstract}
		Let $M$ be a compact hyperbolic $3$-manifold with volume $V$. Let $L$ be a link such that $M\setminus L$ is hyperbolic. For any hyperbolic link $L$ in $M$, in this article, we establish an upper bound of the length of an $n^{th}$ shortest closed geodesic as a logarithmic function of $V$ in $M\setminus L$. Our works complement the work of Lakeland and Leininger \cite{Christopher} on the upper bound of systole length.
	\end{abstract}
	\maketitle
  

\section{Introduction}	
In this article, our primary focus is to give an upper bound on the length of an $n^{th}$ shortest closed geodesics for the hyperbolic link complement of compact hyperbolic $3$-manifolds. We recall that a $geodesic$ in a Riemannian manifold is a locally distance-minimizing curve, and the {\it systole} of a Riemannian $n$ manifold $M$ is the shortest closed geodesic on $M$. Throughout this article, we restrict ourselves to the case at $n=3$.~A finite volume hyperbolic $3$-manifold is a complete  Riemannian $3$- manifold with constant sectional curvature $-1$. We are especially motivated by William Thurston's Geometrization Conjecture \cite{Thu1}, which is proved by Grigory Perelman, based on the program referred by Richard Hamilton\cite{GP,GP2,AR}. Also, William Thurston proved the hyperbolization theorem, known as Thurston hyperbolization theorem, which states that a compact 3-manifold with a nonempty torus boundary has an interior admitting a complete hyperbolic structure if and only if it is irreducible, boundary irreducible, atoroidal, and anannular\cite{Thu1}. As a result William Thurston originated an important classification of knots in $S^3$. The author proved that a knot $K$ in $S^3$ is hyperbolic if and only if $K$ is not a torus knot or a satellite knot\cite{Thu1}. With these celebrated knot classification results in hand, we can intuitively see how to identify non-compact hyperbolic $3$-manifolds. More interestingly, it also produces an infinitely many compact $3$-manifold with a torus boundary whose interior admits a complete hyperbolic metric with finite volume. This hyperbolic $3$-manifolds came basically as a knot complements $S^3$; this are known as cusp hyperbolic $3$-manifolds. In the next section, we will discuss details about cusp hyperbolic $3$-manifolds and their geometric lift, known as the horoball diagram or cusp diagram in the upper half space model. One of the essential studies in Riemannian Geometry is the study of geodesics in Riemannian manifolds. In 1983 Mikhael Gromov \cite{MG} proved that if $M$ is a $n$-dimentional Riemannian manifold then for some universal constant, $M$ admits a closed geodesic $\gamma$ such that 

\begin{equation*}
\text{length}(\gamma) \leq \text{const}_n( \text{Vol}(M))^{\frac{1}{n}}.
\end{equation*}

Let $M$ be compact hyperbolic $3$-manifolds, and $L$ be a hyperbolic link in $M$. Our contribution in this article is to provide an upper bounds on the length of an $n^{th}$ shortest closed geodesics in the hyperbolic knot complements of compact hyperbolic $3$-manifolds.
 
For cusp hyperbolic $3$-manifolds are very special type of Riemannian manifolds where we can study geometric properties of close geodesics by the help of its geometric properties of horoball diagram or cusp diagram. In these article we are mainly focusing give an upper bounds on lengths  of all closed geodesics of hyperbolic knot complements of compact hyperbolic $3$-manifolds by the help of horoball diagram.

We are primarily motivated by the work {\it Systole of hyperbolic $3$-manifolds} by Adams and Reid \cite{Adams2000}, where authors established a universal upper bound on length of the systole of a class of cusped hyperbolic $3$-manifolds which are come from knot complements of $3$-manifolds which do not admit any Riemannian metric of negative curvature. Here, we give the statement of the result by Adams and Reid, which play an essential role in our current context.

\begin{thm}{\cite[Theorem 1.1 ]{Adams2000}} (Adams and Reid){\label{Th1}} If $M$ is a closed orientable 3-manifold that does not admit any Riemannian metric of negative curvature, then for any hyperbolic link $L \subset M$ the systole length of $M\setminus L$ is less than or  equal to $7.3553 \dots$ ~.
\end{thm}
Now we recall that for a given  cusped  hyperbolic 3-manifolds $M$ the cusp density defined as a ratio of volume of total cusp over the hyperbolic volume of the manifold. More precisely,
\begin{align}
\text{the cusp density} =\dfrac{\text{cusp volume}}{\text{volume of the hyperbolic  link complemets}}.
\end{align}

Horoball packing results of B\"{o}r\"{o}czky,\cite{B} it is well known that the cusp density for any non-compact finite-volume hyperbolic 3-manifold or 3-orbifold must lie in the interval $(0, 0.853,\ldots ]$, (see \cite{M, C.Adams2002}), where $0.853 \dots = \dfrac{\sqrt{3}}{2v_0}$, and $v_0=$volume of the regular ideal hyperbolic tetrahedron. The highest cusp density is realize by figure-8 knot complements \cite{B}, which has maximal cusp volume $\sqrt{3}$ and maximal volume $2.02988...$.Therefore for a given any cusp hyperbolic 3-manifold of volume $V$, the cusp volume of that manifold is not more than $C_0V$, where $C_0=0.853 \dots = \dfrac{\sqrt{3}}{2v_0}$.

In this context, we recall the result established by Lakeland and Grant\cite{Christopher} in 2014. Here,the authors gives an upper bound of systole length of proved that, if $M$ is a closed, orientable $3$-manifold with volume $V$, and suppose $L$ is a hyperbolic link in $M$, then the systole length bound of $M\setminus L$ bounded above by $\log((\sqrt2(C_0V)^{2/3}+4\pi^2)^2 +8)$. The author's declared for give a more general statement is that, if $M$ does not admit hyperbolic structure then assume $V=0$. In particular. Now if we put $V=0$ on the above expression, we get the bound for systole length approximately $7.3566..$, which has been already proved  by Adams and Reid Theorem ({\ref{Th1}}) for non-hyperbolic cases. Also, set $C_0= \frac{\sqrt{3}}{2v_0}$, where $v_0$ is the volume of regular ideal tetrahedron.

In this article, $\textup{Vol}(M)$ refers to the volume of the manifold $M$.

\begin{thm}{\cite[Theorem 1.2 ]{Christopher}}\label{LG} If $M$ be a closed orientable $3$-manifold and $V= \textup{Vol}(M)$, then for any hyperbolic link $L\subset M$, the systole length of $M \setminus L$ is given by 
 \begin{equation}\label{LGb}
\log\left[\left(\sqrt2(C_0V)^{2/3}+4\pi^2\right)^2 +8\right].  
 \end{equation}	
\end{thm}
 We recall upper bounds for the length of $n^{th}$ shortest closed geodesics in a hyperbolic knot or link complement in closed $3$-manifolds, where the closed manifolds are does not admit any Riemannian metric of negative curvature \cite{palaparthi2013closed}. From the earlier discussions, it is clear that the upper bounds of $n^{th}$ shortest closed geodesic of hyperbolic link complement of compact hyperbolic $3$ manifolds are not yet known. In this article, we try to focus on generalizing the systole bound given by Lakeland and Grant in (\ref{LG}) to the $n^{th}$ shortest closed geodesics length bound in the knot complement of a compact hyperbolic 3-manifolds. Hereby we state the main result of the article, which is given as follows:

%

\begin{thm}\label{Main Rresult}
If $M$ is a orientable compact hyperbolic $3$-manifold, and $L$ is a hyperbolic link in $M$, and volume of $M$ is $V$.  Let $l_n$ represents the length  $n^{th}$ shortest closed geodesic length in $M\setminus L$, then

\begin{equation}\label{in1}
		l_n \leq \log\left[ n^2\left(\sqrt{2}\left(C_0V\right)^{2/3} + 4\pi^2\right)^{2}+8 \right].
	\end{equation}
	
\end{thm}
\begin{rmrk}
It is important to note that the above upper bound \eqref{in1} for the length of a $n^{th}$ shortest closed geodesic in $M$ is the generalization of upper bound of  the systole length mentioned in Theorem \ref{LG}, in the sense that if we put $n=1$ in \eqref{in1}, we get back the bound \eqref{LGb}.
\end{rmrk}
The methodology of the proof is basically based on some basic geometric observation on the cusp diagram (horoball diagram) of the hyperbolic 3-manifold $M \setminus L$. We are basically producing countably infinitely many  geodesics in the cusp diagram, which are the lifts of the closed geodesics in the manifold $M \setminus L$. Then we find the length of that countably infinite collection of closed geodesics, and finally, we give the upper bounds length of the shortest closed geodesics in the manifolds. The vital ingredient of the new upper bound in (\ref{in1}) is that  new upper bounds are in terms of the volume of a compact hyperbolic 3-manifolds. 

We present some auxiliary results relevant to our discussion in the next section. Equipped these technicalities we prove the main result in the final section.

\section{Preliminaries}
 There are several models of hyperbolic $3$-space. Throughout this article, we will work on the upper half space model. In this model hyperbolic $3$ space is defined to be the the set of points in upper half space i.e. $\mathbb{H}^3=\{(z,t)\in \mathbb{C\times R}: t>0\}$ equipped with the Riemannian metric
\begin{equation}\label{E1}
	ds^2=\dfrac{|dz|^2+dt^2}{t^2}.
\end{equation}
Boundary of $\mathbb{ H }^3$ is defined by $\partial\mathbb{ H }^3=\{(z,t)\in \mathbb{C\times R}:t=0 \}\cup\{\infty\}$. With respect to above metric, the all geodesics in $\mathbb{H}^3=\{(z,t)\in \mathbb{C\times R}: t>0\}$ are vertical lines and semicircle perpendicular to the  boundary of $\mathbb{H}^3$. Complete geodesic planes are all vertical planes and hemispheres orthogonal to the boundary of $\mathbb{H}^3$.

Let us define $Iso^{+}(\mathbb{ H }^3)$ to be the set of all orientation preserving isometry	on $\mathbb{ H }^3$. If we take a M\"{o}bius transformation from $PSL(2,\mathbb{C})$ it can be ex-tented uniquely to a orientation preserving isometry of $\mathbb{ H }^3$. 

Suppose $	\begin{pmatrix}
	a & b \\
	c & d
\end{pmatrix} \in PSL(2,\mathbb{C})$, and the extension of 
$\begin{pmatrix}
	a & b \\
	c & d
\end{pmatrix}$ is $T: \mathbb{ H }^3 \rightarrow \mathbb{ H }^3$

\label{2.2}
\begin{equation}
 T(z,t) =
\begin{cases}
\displaystyle	\left(- \dfrac{\overline{z+d/c}}{c^2(|z+d/c|^2 + t^2)} + \dfrac{a}{c}, \dfrac{t}{c^2(|z+d/c|^2 + t^2)}\right)  & c \neq 0  \\
\left(\dfrac{a}{c}(z+b/a), \left|\dfrac{a}{d}\right|t\right) & c=0
\end{cases}  .
\end{equation}

Every orientation preserving M\"{o}bius transformation has a unique continuous extension from $\mathbb{ H }^3 \cup \partial\mathbb{ H }^3$ to $\mathbb{ H }^3 \cup \partial\mathbb{ H }^3$, whose restriction to $\mathbb{ H }^3$ is an isometry with respect to the above metric $ds^2$. Conversely,every isometry of $\mathbb{ H }^3$ with respect to the same metric we will get this way {\cite[Theorem 9.8 ]{FB}}.

So, without loss any generality we may assume that~ $Iso^{+}(\mathbb{ H }^3)=PSL(2,\mathbb{C})$ and $PSL(2,\mathbb{C})$ acts on the boundary of $\mathbb{ H }^3$ via M\"{o}bious transformations.

A $3$-manifold is said to be hyperbolic if it admits a complete Riemannian metric with $-1$ constant sectional curvature. A subgroup $\Gamma$ of $PSL(2,\mathbb{C})$ is said to be a {\it Kleinian group} if it is a   discrete subgroup of $PSL(2,\mathbb{C})$.

Here we are going to state an important theorem which relates the geometric and topological invariants of complete hyperbolic 3-manifolds.

\begin{thm}{\cite[Theorem 6.1 ]{JP}}(Mostow-Prasad rigidity).
If $M^n_1 $	and $M^n_2$ are complete hyperbolic $n$ manifold with finite volume and $n > 2$, then any isomorphism of fundamental groups $\phi :\pi_1(M^n_1)  \mapsto \pi_1(M^n_2)$ is realized by unique isometry.
\end{thm}
	
Now if $M$ is complete hyperbolic $3$ manifold with finite volume then by Mostow-Prasad rigidity theorem the Riemannian metric with constant $-1$ curvature is unique i.e.  geometric invariants are topological invariants. A subgroup $\Gamma$ of $PSL(2,\mathbb{C})$ is discrete if and only if its action on $\mathbb{ H }^3$ is properly discontinuous.
		
Any hyperbolic $3$ manifold obtained by the quotient $\mathbb{ H }^3/\Gamma$, where $\Gamma$ is torsion-free Kleinian group acting on $\partial \mathbb{ H }^3$ via M\"{o}bius transformations. Let $M= \mathbb{ H }^3/\Gamma$, be a hyperbolic $3$-manifold, then $\Gamma$ is isomorphic to $\pi_1(M)$. Let $M= \mathbb{ H }^3/\Gamma$ be $3$-manifold which is non compact and finite volume. Let $M$ consist finite number of cusps which is homeomorphic to $T^2 \times [0.\infty)$. Lifting of the cusp in $\mathbb{ H }^3$ is a countable collection of disjoint horoballs covering the cusp in $\mathbb{ H }^3$. Horoballs are equivalent under the action of $\Gamma$ on $\mathbb{ H }^3$.

A hyperbolic $3$-manifold containing cusps is called a cusp hyperbolic $3$-manifold. Let $M$ be a cusp hyperbolic $3$-manifold. Now we are thickening the cusp continuously until they touch themselves. Therefore we can not thicken it further; this is called the maximal cusp. So, A cusp is said to be a maximal if there is no larger cusp containing it which occurs exactly when the cusp is tangent to itself at one or more points. Lift of the maximal cusp of $M$  is the countably infinite set of horoball in $\mathbb{ H }^3$ with disjoint interiors and some point of tangency on their boundaries. Each horoball touches the boundary of the upper half space at a single point. These points are called the center of the horoballs. Projection of the interior of the hroballs to $M$ are tubular neighborhoods of the cusp in $M$. We choose the point at $\infty$ as the centre of horoball denoted $\mathbb{H}_{\infty}$ that covers the maximal cusp. We will normalize so that the boundary of $\mathbb{H}_{\infty}$  is the plane $t=1$ and the projection region $t>1$ tubular neighborhoods of the cusps in $M$. The horoballs whose center at $\partial\mathbb{ H }^3$  is tangent to the plane $t=1$, are called full-sized horoballs.

Now for a given cusp hyperbolic 3-manifold in the maximal cusp diagram, there must exist an isometry $\gamma$ in $\Gamma$ such that $\gamma(H_0)=H_{\infty}$. In $\partial \mathbb{ H }^3$ i.e. in $\mathbb{C}\cup\{\infty\}$, $\gamma(0)=\infty$. If we extended this $\gamma$ to the M\"{o}bious transformation in $\mathbb{ H }^3$ via the above formula (\ref{2.2}), we get that radius of the corresponding isometric circle of $\gamma$ is 1. So, the isometry must be of the form
\begin{align}
\gamma=\begin{pmatrix}
a & -\frac{1}{c} \\
c & 0
\end{pmatrix} \text{, where $|c|=1$} ~.
\end{align}

\subsection{Geodesic Spectrum}
Let $M$ be finite volume hyperbolic $3$ manifold. Then there exists a torsion-free Kleinian group $\Gamma$  such that $M$ is homeomorphic to $\mathbb{ H }^3/\Gamma$. Also $\Gamma$ is isomorphic to $\pi_1(M)$. As $\Gamma$ is discrete so $\pi_1(M)$ is countably infinite. Therefore, each free homotopy loop class in $\pi_1(M)$ contains a unique closed geodesic representative. Hence the number of closed geodesic is countable.
The lengths of closed geodesics in $M$ forms a discrete subset of $\mathbb{R}$. Let $L=\{l_1, l_2, \dots, l_n,\dots\}$ be the set of all lengths of closed  geodesics in $M$. As $L$ is discrete we can arrange the length spectrum by increasing order, like $l_1 \leq l_2 \leq,\dots\leq l_n,\dots$. Thus the systole length of $M$ is $l_1$.

Let $M$ be complete hyperbolic $3$ manifold with at least one cusp. Let $M=\mathbb{ H }^3/\Gamma$ where $\Gamma$ is discrete subgroup of $PSL(2,\mathbb{C})$ isomorphic to $\pi_1(N)$. We can conjugate $\Gamma$ such a way that the stabilizer of $\infty$ which we denote $\Gamma_{\infty}$ will be parabolic subgroup of $\Gamma$ which is isomorphic to $\mathbb{Z} \oplus \mathbb{Z}$. If $l$ is the minimal translation length we can always take one generator of
$\Gamma_{\infty}$ is

\begin{center}
	$
	\beta=\begin{pmatrix}
		1 & l \\
		0 & 1
	\end{pmatrix}. $
	
\end{center}

In this ar we will always assume that $0$ $\in \partial\mathbb{ H }^3$ and $\infty$   $\in \partial\mathbb{ H }^3$ are the parabolic fixed points of $\Gamma$.


\subsection{Dehn filling}	


Let $M$ be a $3$-manifold with torus boundary component $\partial M$ homeomorphic to torus $T^2$, therefore $\pi_1(\partial M)$ isomorphic to the group $\mathbb{Z} \oplus \mathbb{Z}$. Let $a$ and $b$ be two generators of $\pi_1(\partial M)$ i.e. $\pi_1(\partial M)= \langle a, b \rangle$, let $s$ be an isotopy class of essential simple closed curve in the boundary torus $\partial M$, so $s= a^m b^n$, $s$ also called slope on $\partial M$. Now let $P=S^1 \times D^2$ a solid torus and $\mu$ be the meridian of $\partial P$, we from a closed $3-$ manifold by $s$-\textbf{Dehn filling} denoted by $M(s)$ by attaching $P$ on $M$ identifying $\phi : \partial P \rightarrow \partial M$ such that $\phi(\mu)=s$. More precisely, $\displaystyle M(s)= {M \sqcup P} / {x \sim \phi(x)}$.
Suppose $M$ has $k$ disjoint torus boundary components, let $s_1, s_2, \dots, s_k$ are the slopes corresponding boundary components, Dehn fill along the slope $s_1, s_2,\dots,s_k$ denoted by $M(s_1, s_2,\dots,s_k)$.

Let $M$ be a cusp hyperbolic $3$ manifold with cusp $C$ then the boundary of $C$ admits a Euclidean Structure, $s$ be any essential simple closed curve in the boundary of $C$ then $s$ is isotopic to a geodesic with well defined Euclidean length such a length is called $slope~ length$.

Let $M$ be a cusp hyperbolic $3-$ manifold with cusps $C_1, C_2,\dots.,C_k$. It is not always possible that after Dehn Filling the manifold $M(s_1,s_2,\dots,s_k)$ will be closed hyperbolic 3-manifold, when it will hyperbolic depends on the slope lengths of $s_1, s_2,\dots,s_k$. The following theorem provides a sufficient condition when $M(s_1,s_2,\dots,s_k)$ is hyperbolic which is known as Gromov–Thurston $2\pi$-theorem, theorem stated as follows:

\begin{thm} {\cite[Theorem 2.1 ]{Adams2000}}($2\pi$-$Theorem$)
Let $M$ be a cusped hyperbolic 3-manifold with n cusps. Let $ C_1,C_2,\dots,C_k $be disjoint cusp tori for the $n$ cusps of $M$, and $s_i$ a slope on $C_i$ represented by a geodesic $\alpha_i$ whose length in the Euclidean metric on $C_i$ is greater than $2\pi$, for each $i \in \{1,2,\dots.,k\}$ Then $M(s_1,s_2,\dots,s_k)$ admits a metric of negative curvature.
\end{thm}    

This slope length bound in $2\pi$ theorem was later improved from $2\pi$ to $6$ by Ian Agol\cite{Agol} and Marck Lacknby \cite{lackenby1998word} independently.

\vspace{.2In}

In {\cite[Theorem 6.5.6]{Thu2}}, William Thurston proved the following theorem.
\begin{thm} \label{Thu2}
	If M is hyperbolic 3-manifold with cusp $C$, and $s$ is a slope on $\partial C$ such that $M(s)$ is hyperbolic, then
	
\begin{align*}
\textup{Vol(M)} > \textup{Vol(M(s))}.
\end{align*}	
	
\end{thm} 

In the final section, we will use the above theorem while explaining our main result(\ref{Main Rresult}). 

\vspace*{.2in}

Let $M$ be  hyperbolic $3$- manifold with finitely many cusps torus, then the lower bound  of the volume of hyperbolic Dehn filled manifolds is given by following theorem of D.Futer et al.\cite{FDKE}.

\begin{thm}{\cite[Theorem 1.1 ]{FDKE}} \label{Futer}
	
Let $M$ be a complete, finite volume hyperbolic manifold with cusps. Suppose $C_1,C_2,\ldots, C_n$ are disjoint embedded cusps with slopes$ s_j $ on $C_j$ such that a geodesic representative of $s_j$ on $\partial C_j$ has length strictly greater than $2\pi$. Denote the minimal slope length $l_{min}$ then the Dehn filled manifold $M(s_1, s_2,\dots,s_n)$ is hyperbolic and with 
	
\begin{align*}
\textup{Vol} \left(M(s_1, s_2,\dots,s_n)\right) \geq \left( 1-\left(\dfrac{2\pi}{l_{min}}\right)^2 \right)^{\frac{3}{2}}  \textup{Vol}(M).
\end{align*}

\end{thm}

We finished this section by providing  two theorems by Lakeland and Leininger \cite{Christopher}. Later we will use this theorems to prove in our main result,
theorems stated as follows:

\begin{thm}{\cite[Lemma 3.1 ]{Christopher}} \label{th23}
	A loxodromic element of $PSL(2, \mathbb{C})$ with trace of modulus bounded above by $R$ has translation length at most $log(R^2 + 4)$.	
\end{thm}

\begin{thm}{\cite[Lemma 3.3 ]{Christopher}}\label{th3}
	Suppose the Kleinian group $\Gamma$ is torsion-free,
	and $\Gamma_{\infty}$ has co-area at most $V_c$ when viewed as acting by translation in $\Gamma_{\infty}$ on $\mathbb{C}$~(or equivalently acting on $H_{\infty}$), and that the minimal parabolic translation length $l>2\pi$, Then $2\pi < l \leq \displaystyle \sqrt{\dfrac{4V_c}{\sqrt3}}$.	
\end{thm}

Now we have enough information to proof our main theorem. So, in the next section we are going to proof our main result.

\section{Proof of the Main Result}
Before going to the main result, we will prove some lemmas. Using these lemmas, we will prove our main result. The following lemmas are proved using some basic geometric observations of horoball diagram of cusp hyperbolic 3-manifolds.

\begin{lma}\label{3.1}
Let $\Gamma$ be torsion-free Kleinian group with a parabolic subgroup $\Gamma_{\infty}$ which fixes the point at $\infty$ and isomorphic to $\mathbb{Z}\oplus\mathbb{Z}$. We can view $\Gamma_{\infty}$ acting on $H_{\infty}$ such that $H_{\infty}/\Gamma_{\infty}$ is homeomorphic to $T^2$, and let the minimal parabolic translation length $l > 2\pi$. Let $\Gamma_{\infty}$ has co-area at most $V_c$ when we view $\Gamma_{\infty}$ acting on $ H _{\infty}$.  Also let $\Gamma$ contains an elements of the from 
\begin{align}
\gamma=\begin{pmatrix}
a & -\frac{1}{c} \\
c & 0
\end{pmatrix} 
\text{, where $|c|=1$, and $|a|$ is the minimal for all $\gamma \in \Gamma$ with $\gamma(0)={\infty}$},
\end{align}
then one of the flowing are true :

\begin{enumerate}
	
	\item if $|a| > 2$, then for each $n \geq 1$ there exist a loxodromic element $\alpha_n$ such that
	 \begin{align}
	 |tr(\alpha_n)| \leq \sqrt{{\left(n-\frac{1}{2}\right)^2} l^2 + \frac{V_c^2}{l^2}}
	 \end{align}
	\item if $|a| \leq 2$, then for each $n \geq 1$ there exist a loxodromic element $\alpha_n$ such that 
	
	\begin{align}
	|tr(\alpha_n)| \leq \sqrt{l^2n^2 +4}
	\end{align}
	
\end{enumerate}
\end{lma}

\begin{proof}
Before we go to the main proof $ 1^{st}$, we look at some observations. As given in hypothesis $V_c$ is the maximal cusp area then it implies that the cusp volume is $\frac{V_c}{2}$. We recall the {\it Dirichlet domain} of the group group of isometries $\Gamma$ at the point $p_0 \in \mathbb{ H }^3$ (here $\Gamma$ acting on $\mathbb{ H }^3$), is defined by $\Delta_{\Gamma}(p_0)=\{p \in \mathbb{ H }^3: d_{hyp}(p,p_0)\leq d_{hyp}(p,\gamma(p_0)) \forall \gamma \in \Gamma\}$. Let $\Delta$ be the {\it Dirichlet domain} of $\Gamma_{\infty}$ at the point $0$ by hypothesis $|a|$ is the minimal trace modules all $\gamma \in \Gamma$, all $\beta \in \Gamma_{\infty}$, $\beta\gamma(0)=\infty$, so $\frac{a}{c} \in \Delta$.
\vspace*{.3cm}

\textbf{Proof~of~(1)}:
Given that $l$ is the minimal translation length corresponding one  parabolic basis in $\Gamma_{\infty}$. We can present this parabolic isometry in $\Gamma$ by the matrix 
\begin{align}
\beta=\begin{pmatrix}
1 & l \\
0 & 1
\end{pmatrix}.
\end{align}
and 
\begin{align}
\beta^{\pm n}\gamma=\begin{pmatrix}
a \pm ncl & -\frac{1}{c} \\
c & 0
\end{pmatrix} \text{, where $|c|=1$}
\end{align}
Now we claim that $\beta^{-n} \gamma$ and $\beta^n \gamma$ both can not be parabolic together, if there exist a natural number $k$ such that $\beta^{k} \gamma$ and $\beta^{-k} \gamma$ then $tr(\beta^{-k}\gamma)=a-ncl$ is $2$ or $-2$ and $tr(\beta^k \gamma)=a+ncl$ is $2$ or $-2$.

\begin{align*}
\text{Now if},~ 
& a-kcl=2\\
& a+kcl=2\\
&\implies kcl=0,\text{which is a contradiction}.\\
\text{Now if},~ 
& a-kcl=-2\\
& a+kcl=-2\\
&\implies kcl=0, \text{which is also a contadiction}.\\
\text{Now if},~ 
& a-kcl=2\\
& a+kcl=-2\\
&\implies a=0, \text{which is a contradiction because $\Gamma$ does not contain any elliptic element}.\\
\end{align*}
Therefore, as a consequence we get that for each $n \geq 1$ one of the $\beta^{-n} \gamma$ or $\beta^n \gamma$ must be loxodromic. So, we can construct a sequence of loxodromic $\alpha_n$ by choosing  one between $\beta^{n} \gamma$ and $\beta^{-n} \gamma$ for all $n \geq 1$. Therefore we can produce a sequence of loxodromic elements whose trace modules either $|a-ncl|$ or $|a+ncl|$. 

By the above discussion we can construct a sequence of distinct loxodromic element $\alpha_n$ such that $\alpha_n^{-1}(0)=0$ and $\alpha_n(\infty)= \frac{a}{c} + ncl$ or $a-ncl$, now trace modules  of $\alpha_n$ is equal to the distance between the center of isometric circle of $\alpha_n$ and $\alpha_n^{-1}$. Let $\Delta$ be the {\it Dirichlet domain} for $\Gamma$ at $0$ and  $\gamma(\infty)= \frac{a}{c}$ is inside the {\it Dirichlet domain}. Now for each $n$ the center of the isometric circles of $\alpha_n^{-1}$ is $\frac{a}{c}+nl$ or $\frac{a}{c}-nl$ and center of the isometric circle of $\alpha_n$ is $0$. So, for each $n$ the maximum distance of the center between  isometric circles $\alpha_n^{-1}$ and $\alpha_n$ given by the coordinate $-(n-\frac{1}{2}){l}+i\frac{1}{2}\frac{2V_c}{l}$  or $(n-\frac{1}{2}){l}+i\frac{1}{2}\frac{2V_c}{l}$.
\vspace*{1.3cm}

\begin{figure}[htbp!]
	\centering
	\includegraphics[width=0.9\textwidth]{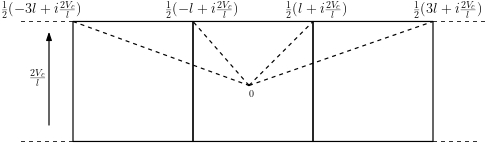}
	\caption{The doted line are position vectors of center of isometric circles and modules of this position vectors bound the trace modules of loxodromic elements $\alpha_n$. }
\end{figure}


Therefor, we get
\begin{align}
|tr(\alpha_n)| \leq \sqrt{{\left(n-\frac{1}{2}\right)^2} l^2 + \frac{V_c^2}{l^2}}
\end{align} 

 \textbf{Proof~of~(2)}:To prove this 2nd part  we consider the sequence of distinct elements  $\beta^{n}\gamma$, for all $n \geq 1$ where the product
 \begin{align}
 \beta^{n}\gamma=\begin{pmatrix}
 a + ncl & -\frac{1}{c} \\
 c & 0
 \end{pmatrix} \text{, where $|c|=1$}.
 \end{align}
Now, $|tr(\beta^{\pm n}\gamma)|=|a\pm ncl|\geq |nl|-|a|>2$ (as $l > 2\pi$ and $|a|\leq 2$), therefore  $\beta^{n}\gamma$ are loxodromic all $n \geq 1$.
(here we note that also one can take $\beta^{-n}\gamma$ all $n \geq 1$)
 

Set $\alpha_n=\beta^n\gamma$,  the center of isometric circles of $\alpha_n$ and $\alpha_n^{-1}$
are receptively $\alpha_n^{-1}(\infty)=0$ and  $\alpha_n(\infty)=\frac{a}{c}+nl$. In this case we are shifting the center center of isometric circles by $nl$ distance. Now trace of $\alpha_n$ is the distance between the points $\alpha_n^{-1}(\infty)=0$ and $\alpha_n(\infty)=\frac{a}{c}+nl$.  
As $|\frac{a}{c}|\leq 2$, by applying Pythagoras theorem we can bound the trace modules bound of $\alpha_n$ for each $n \geq 1$.

\begin{figure}[htbp!]
	\centering
	\includegraphics[width=0.8\textwidth]{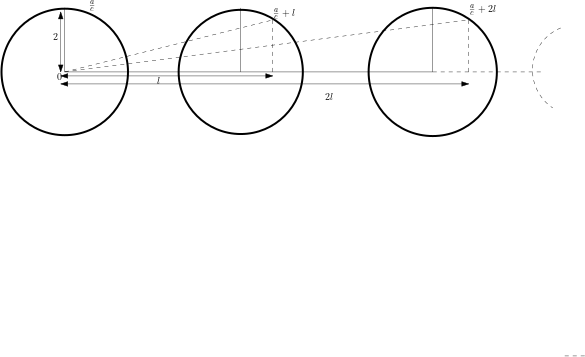}
	\caption{Trace modules of $\beta^n\gamma$ less than $\sqrt{l^2n^2+4}$. }
\end{figure}

 Therefore, $|tr(\beta^{n}\gamma)|\leq \sqrt{l^2n^2+4}$.

\end{proof}

\begin{lma}\label{th*}
	Let $M$ be a finite volume hyperbolic 3 manifold with at least one cusp and let $M=\mathbb{H}^3/\Gamma$ where $\Gamma$ discrete subgroup of $PSL(2,\mathbb{C})$. Assume that in a maximal cusp torus, there is non trivial curve corresponding to a parabolic isometry of length $l> 2\pi$. Then each $n \geq 1$ we can bound trace modules  of a loxodromic elements by $\displaystyle\sqrt{l^4n^2 +4}$.
\end{lma}	

\begin{proof}
	Let $M=\mathbb{ H }^3/\Gamma$, $\Gamma$ is the image of faithful representation of $ \pi_1(M,b)$ where $b$ is base point at the point of tangency of maximal cusp in $M$ and  $\Gamma$ is  a discrete subgroup of $PSL(2,\mathbb{C})$.Let $H_0$ and $H_{\infty}$ are  horoballs whose center at $(0,0,0)$ and infinity. We arrange maximal cusp so that point of tangency between the horoballs $H_0$ and $H_{\infty}$ at $(0,0,1)$. The inverse image of $b$ under covering projection contain the point $(0,0,1)$. Let $x_1$ be the lift of meridian curve $c$ (base at $b$) to the horosphere $H_{\infty}$ with beginning point at $(0,0,1)$ and final end point at $(l,0,1)$. 
	
	\begin{figure}[htbp!]
		\centering
		\includegraphics[width=0.5\textwidth]{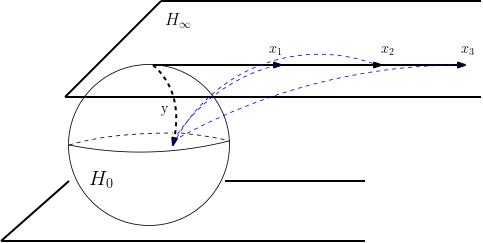}
		\caption{$H_0$ horoball at 0, $H_{\infty}$ horoball at $\infty$ and blue doted lines are geodesic passing through the end points of the lifts $x_n$ and $y$. }
	\end{figure}
There is another lift $y$ of $c$ on the  boundary of the full size  horoball $H_0$. In this lift $y$ having same starting point $(0,0,1)$ and end point some $p'$ on the boundary of $H_0$. Let $\beta$ and $\gamma$ be the corresponding parabolic isometries in $\Gamma$  of the paths $x_1$ and $y$. Now we can conjugate $\Gamma$ such that  $0$ and $\infty$ parabolic fixed points of $\beta$ and $\gamma$. So, we can take the isometries $\beta$ and $\gamma$ as follows :
\begin{align}
\beta^{-1} & =
\begin{pmatrix}
1 & l\\
0 & 1 
\end{pmatrix} 
\text{,where l is the minimal parabolic length} \\
&\text{and}\\
&\gamma  =
\begin{pmatrix}
1 & 0\\
\omega & 1 
\end{pmatrix}
\text{,where $\omega$ is a complex number}
\end{align}	
Now we consider the products for $n \geq 1$ 
\begin{center}
	$d_{\pm n} =$
	$
	({\beta^{-1}})^{\pm n} \gamma =
	\begin{pmatrix}
	1\pm nl\omega &  nl \\
	\omega & 1
	\end{pmatrix} $.		
\end{center} 	
 Then, for every $n \geq 1$, $d_{-n}$ and $d_n$, both can never parabolic together. So, for each $n$ one of the $d_{-n}$ or $d_n$ must be loxodromic. Therefore for each $n$, we can always get a loxodromic element whose trace modules are either $|2+ nl\omega|$ or $|2- nl\omega|$.
 
 Now let $x_n$ be the lift corresponding to the parabolic elements $\beta ^n$  If the angle between $x_n$ and $y$ is $\dfrac{\pi}{2}$ and if the product $\displaystyle ({\beta^{-1}})^n \gamma$ is loxodromic then the geodesic lengths corresponding the products is maximum. As we give an upper bound on shortest $n^{th}$ closed geodesics, without loss of any generality we can assume that the absolute value of the angle between $x_n$ and $y$ at $(0,0,1)$ is $\pi/2$.  So, we can choose the isometry $\beta^{-1}\gamma$ when $x$ and $y$ perpendicular we can represent the parabolic isometries $\beta$ and $\gamma$ by
	
\begin{center}
		$\beta^{-1} =
		\begin{pmatrix}
			1 & l\\
			0 & 1 
		\end{pmatrix}$
		and $\gamma=				
		\begin{pmatrix}
			1 & 0\\
			il & 1 
		\end{pmatrix}$
	\end{center}

	Now for each $n \in \mathbb{N}$,~ $\displaystyle ({\beta^{-1}})^n \gamma$ are distinct  and all are loxodromic as $\displaystyle|tr(\beta^{-1})^n \gamma)| >2$. Now
		
	\begin{center}
		$\alpha_n =$
		$
		({\beta^{-1}})^n \gamma =
		\begin{pmatrix}
			1+inl^2 &  nl \\
			il & 1
		\end{pmatrix} $. 
	\end{center} 
Therefore $\displaystyle |tr(\alpha _n)| = 
	\displaystyle \sqrt{l^4n^2 +4}$, hence for each $n \geq 1$ we can give an upper bound of trace modules  of loxodromic elements by $\displaystyle\sqrt{l^4n^2 +4}$.	
\end{proof}
Here we note that closed geodesic corresponding the loxodromic elements $({\beta^{-1}})^n \gamma$ all are distinct so we can rank the length spectrum of the corresponding closed geodesic by $l_1 \leq l_2 \leq l_3\dots\dots$.	Where $l_1$ is the systole length, $l_2$ is the $2^{nd}$ shortest closed  geodesic length,  $l_3$ is the $3^{rd}$ shortest closed  geodesic length $\dots\dots$ ~.


%
%
%
%
Next lemma is a generalization of lemma {\cite[Lemma 3.3 ]{Christopher}}. Techniques of the proof of this  lemma has been taken from \cite{Christopher}.
\begin{lma} \label{3.3}
Let $N=\mathbb{H}^3/\Gamma$ be a non compact hyperbolic $3$ manifold with finite volume where $\Gamma$ is a discrete subgroup of $PSL(2,\mathbb{C})$ and $N=M\setminus L$ where $M$ is a closed orient-able $3$ manifold and $L$ is hyperbolic link in $M$ and $\textup{Vol}(N)=V$. Then for each $n\geq1$, $\Gamma$ has loxodromic element with trace modules less than or equal to

\begin{align*}
\sqrt{2n^2V_c^{\frac{4}{3}} +4},\text{ where $V_c=C_0V$.} 
\end{align*}

\begin{proof}
	Let us consider the sequence of  functions $S_n(l)=min{T_n(l), AR_n(l)}$,
	
	\vspace*{.299cm}
	where 
\begin{align}
&T_n(l)=\sqrt{{\left(n-\frac{1}{2}\right)^2} l^2 + \frac{V_c^2}{l^2}}, \text{ and}	\\
&  AR_n(l)=\sqrt{n^2l^4+4} ~\text{where}, 2\pi < l \leq \sqrt{\frac{4V_c}{\sqrt 3}}	
\end{align}	
to make the proof easy we choose the sequence of functions
\begin{align}
&s_n(x)=min\{a_n(x),t_n(x)\},
\end{align}
where
\begin{align*}
& a_n(x)={n^2x^2+4},~ \\
& t_n(x)= {{\left(n-\frac{1}{2}\right)^2} x + \frac{V_c^2}{x}}
\end{align*}

  and, as $4\pi^2 > 1$, without loss any generality we can assume that $1 < x \leq {\frac{4V_c}{\sqrt 3}}$.
  
 Let $$\displaystyle f_n(x)=a_n(x)-t_n(x)=\dfrac{n^2x^3+4x-(n-\frac{1}{2})^2x^2-V_c^2}{x},~ \text{where}~ 1 < x \leq {\frac{4V_c}{\sqrt 3}}.$$
 
 Let us assume that
\begin{align}
 \displaystyle g_n(x)= n^2x^3+4x-(1-\frac{1}{2})^2x^2-V_c^2,
\end{align}
 the derivative 
\begin{align}
 g'_n(x)=3n^2x^2 -2(n-\frac{1}{2})^2x +4,
\end{align}
and $g'_n(x) > 0$ each $n \geq 1$. For each $n \geq 1$ therefore $g_n$ is an strictly incising function and also $g_n(0)<0$, so $g_n$ has unique a real root call it $x_n$. As leading coefficient of each $g_n$ is positive this implies that $x_n >0$. 
  
Therefore, for every $n \geq 1$, and $x > x_n$,
\begin{align*}
~~ g_n(x) >0 
 \implies a_n(x) > t_n(x),\quad \forall\, n.
\end{align*}
Now, we claim that for $n\geq1$, $x_n < \sqrt{2}V^{\frac{2}{3}}_c$.

Let us consider 
\begin{align*}
&a_n(\sqrt{2}V^{\frac{2}{3}}_c)-t_n(\sqrt{2}V^{\frac{2}{3}}_c)\\
&=2n^2V^{\frac{4}{3}}_c-(n-\frac{1}{2})^2\sqrt{2}V^{\frac{2}{3}}_c-\frac{1}{\sqrt 2}V^{\frac{4}{3}}_c +4\\
&=(2n^2-\frac{1}{\sqrt{2}})V^{\frac{4}{3}}_c-(n-\frac{1}{2})^2\sqrt{2}V^{\frac{2}{3}}_c+4 \\
& \geq V^{\frac{2}{3}}_c\left(\left(2n^2-\frac{1}{\sqrt 2}\right)-\sqrt{2}n^2+\sqrt{2}n-\frac{\sqrt{2}}{4}\right)+4 \geq 0
~~(\text{ as $V_c > 1$}).
\end{align*}
 This flows the claim. Therefore $a_n(x_n)<a_n(\sqrt{2}V^{\frac{2}{3}}_c)$.
 Again, there are two possibilities either $x_1 \leq 1$ or $x_1 > 1$, also one can easily  verify that $x_n \leq x_1$ all $n \geq 1$. Now if $x_1 \leq 1$ then, $a_n(x)>t_n(x)$ for all $x\in(1,\frac{4V_c}{\sqrt{3}}]$.
 
 Therefore,
 \begin{align*}
& s_n(x)=t_n(x), ~\text{for all}~x\in \left(1,\frac{4V_c}{\sqrt{3}}\right] ~\text{and for all}      ~ n\geq 1.\\
\implies &s_n(x) \leq a_n(\sqrt{2}V^{\frac{2}{3}}_c)~\text{for all}~x\in\left(1,\frac{4V_c}{\sqrt{3}}\right] ~\text{and for all}      ~ n\geq 1
& (\text{ as }~\sqrt{2}V^{\frac{2}{3}}_c > 1)
\end{align*}
 For worst case we may assume that $x_n \in \left(1,\frac{4V_c}{\sqrt{3}}\right] $ for all $n \geq 1$
\begin{align*}
&s_n(x)=a_n(x), ~\text{for all}~x\in(1,x_n] ~\text{and for all}      ~ n\geq 1.\\
&s_n(x)=t_n(x), ~ \text{for all}~x\in\left(x_n,\frac{4V_c}{\sqrt{3}}\right] ~\text{and for all}      ~ n\geq 1.
\end{align*}

Now for each $n \geq 1 \text{ and for all}~x\in(1,x_n] $,
\begin{align}
s_n(x)=a_n(x) \leq a_n(x_n)  \leq a_n(\sqrt{2}V_c^{\frac{2}{3}}) ~ (\text {As $x_n< \sqrt{2} V_c^{\frac{2}{3}}$})
\end{align}

For each $n\geq1$ $t_n$ is a convex function. The minimum of $t_n$ for each $n$ at the point $\dfrac{2V_c}{2n-1}$. As $n$ increasing the value of the $\dfrac{2V_c}{2n-1}$ is decreases. So, there are two possibilities either the minimum value of $t_n$ belongs to $\left(x_n,\frac{4V_c}{\sqrt{3}}\right]$ or not in the interval $\left(x_n,\frac{4V_c}{\sqrt{3}}\right]$. Both of this cases 

\begin{align*}
&t_n(x) \leq~ max \left \{t_n(x_n),t_n\left(\frac{4V_c}{\sqrt{3}}\right) \right   \} ~ \text{all} ~n \geq 1.
\end{align*}
 Now one can easily verify that $t_n(x) \leq a_n(\sqrt{2}V_c^{\frac{2}{3}})$ all $n \geq 1$ in $\left(x_n,\frac{4V_c}{\sqrt{3}}\right]$.
\end{proof}
\end{lma}

\begin{rmrk}
Also, we have an upper bound of loxodromic elements from the case of the lemma (\ref{3.1}), which is $r_n(l)=\sqrt{l^2n^2+4}$ for each $n \geq 1$. This is an increasing function of $l$  for each $n \geq 1$ and  $2\pi < l \leq \displaystyle \sqrt{\dfrac{4V_c}{\sqrt3}}$. So, $r_n(l) \leq \sqrt{\frac{4V_c}{\sqrt{3}}n^2+4}$ for each $n \geq 1$ and by the above lemma (\ref{3.3}) we also have a trace modules bound of $n^{th}$ loxodromic elements for each $n \geq 1$.  Now,
\begin{align*}
\sqrt{\frac{4V_c}{\sqrt{3}}n^2+4} \leq \sqrt{2n^2V_c^{\frac{4}{3}} +4} \iff V_c \geq \dfrac{8}{3\sqrt{3}}
\end{align*}
and already $V_c \geq {\pi^2 \sqrt{3}}$. Therefore the bound already we have from the lemma \ref{3.3} is the best bound.
\end{rmrk}

\vspace*{.6cm}
Now set $M\setminus L=\mathbb{ H }^3/\Gamma_{L}$ and $x~=~\textup{Vol} (M\setminus L)$, for the different hyperbolic link $L$ in $M$ we get different $x$. Therefore, the above expression from \ref{3.3}, which is trace modules of loxodromic elements in terms of volume of $M\setminus L$, will produce a sequence of functions of $x$.  Hence the sequence of functions is $\displaystyle F_n(x)=\sqrt{{2n^2C^{4/3}_  0x^{4/3}}+4}$, which is increasing for each $n\geq 1$ and $2\pi < x \leq \sqrt{\frac{4V_c}{\sqrt 3}}$.

To get another loxodromic trace modules optimization sequence function from the following theorem by D. Futer et al.(\ref{Futer}).

\begin{thm}
Let $M\setminus L$, be a complete, finite volume, non-compact hyperbolic manifold. Let $M$ denote the closed manifold obtained by filling along specified slopes on each of boundary tori, each of which length at least $2\pi$, and the least of which is denoted  by $l_{min}$. Then $M$ is hyperbolic, and 
	\begin{center}
		$\textup{Vol}(M)\geq\left(1-\left(\frac{2\pi}{l_{min}}\right)^2\right)\textup{Vol}(M\setminus L)$.
	\end{center}
\end{thm}

By solving the inequality and using the fact given in the theorem we get
 
\begin{align}
l_{min}\leq \dfrac{2\pi}{\sqrt{1-\left( \dfrac{V}{\textup{Vol}(M \setminus L)}\right)^{2/3}}}.\label{3.4}
\end{align}

Thus the minimal parabolic length bounded by a function of $x$ with $x > V $ (by (\ref{Thu2})) , we have

Therefore, finally we get
\begin{align}
	l_{min}\leq \dfrac{2\pi}{\sqrt{1-\left( \dfrac{V}{x}\right)^{2/3}}} 
\end{align}

Now apply the Lemma (\ref{th*}) for each $n$ we can bound trace modulus of loxodromic elements in $\Gamma_{L}$ by $\displaystyle \sqrt{l^4n^2 +4}$. 
Thus  trace modulus also bounded by the decreasing functions.

\begin{align}
	G_n(x)=\sqrt{\frac{16\pi^4 n^2}{(1-(V/x)^{2/3})^2}+4}.
\end{align}

%

Now, we are ready to present the proof for the main result, i.e. Theorem (\ref{Main Rresult}) in the next section.

\subsection{Proof of Theorem (\ref{Main Rresult})}.

According to the above discussion, we got two set of sequence functions, for every natural number $n$
	
	\begin{align}
	\displaystyle F_n(x)=\sqrt{{2n^2C^{4/3}_  0x^{4/3}}+4},
	\end{align}
	and
	
	\begin{align}
		G_n(x)=\sqrt{\frac{16\pi^4 n^2}{(1-(V/x)^{2/3})^4}+4}.
	\end{align}
	
Now observe that $ F_n(x)$ are increasing and $ G_n(x)$ are decreasing functions for each $n$. If there is a common value of these two functions, each $n$ value should be unique. If we plug in the common value in $ F_n(x)$, (As $ F_n(x)$ is increasing for each $n$) this gives an upper bound of the trace modulus of the loxodromic elements for each $n$. Now for each $n$ equating both of them we get 
	
	\begin{align*}
		&F_n(x) =G_n(x)\\
		&\implies  x  = \left(V^{2/3}+\dfrac{4{\pi }^2}{\sqrt{2}C_0^{2/3} }\right)^{3/2}
	\end{align*}
	
Now by applying Theorem (\ref{th23}), we get the $n^{th}$ geodesic bound and the bound is given by
	
\begin{align*}
	l_n \leq \log({F_n(x)}^2+4)=&\log\left[2n^2 C_0^{4/3}\left(V^{2/3}+\dfrac{4\pi^2}{\sqrt{2}C_0^{2/3}}\right)^{2} + 8 \right]\\
	 = &\log\left[ n^2\left(\sqrt{2}\left(C_0V\right)^{2/3} + 4\pi^2\right)^{2}+8 \right]. 
\end{align*}
	
$\hfill\square$

\section{Acknowledgment}
The author would like to thank the Department of Mathematics, Indian Institute of Technology Guwahati for its support.



\begin{thebibliography}{99}
\bibitem{Adams2000}
	Adams, Colin C.; Reid, Alan W.: {\it Systoles of hyperbolic {$3$}-manifolds},  Math. Proc. Cambridge Philos. Soc. {\bf 128}, no. 1 (2000), 103--110.
	
  \bibitem{C.Adams2002}
 Adams, Colin,{\it Cusp densities of hyperbolic 3-manifolds},Proc. Edinb. Math. Soc. (2).  {\bf45}, no. 2 (2000), 277--284.	
 
\bibitem{Agol}
Agol, Ian, {\it Bounds on exceptional Dehn filling}, Geometry and Topology.  {\bf4} (2000), 431--449.
  
\bibitem{GP}
Bruce Kleiner and John Lott, Notes on Perelman’s papers, Geom. Topol. 12 (2008),
no. 5, 2587–2855.

\bibitem{FB}
Bonahon, Francis {\it Low-dimensional geometry,From Euclidean surfaces to hyperbolic knots},{American Mathematical Society, Providence, RI; Institute for
	Advanced Study (IAS), Princeton, NJ}.{\bf 49},no.2866931(2009), xvi+384
\bibitem{B}
B\"{o}r\"{o}czky, K (1978). {\it Packing of spheres in spaces of constant
	curvature}, {Acta Math. Acad. Sci. Hungar.}. {\bf32}, no. 3-4 (1978), 243--261.


\bibitem{FDKE}
Futer, David and Kalfagianni, Efstratia and Purcell, Jessica S,{\it Dehn filling, volume, and the {J}ones polynomial},Journal of Differential Geometry,  {\bf78} (2008), no. 3, 429–464.
\url{http://projecteuclid.org/euclid.jdg/1207834551}

\bibitem{AR}
de Freitas, Izabella and Ramos,{\it Geometrization in geometry},{Mat. Contemp.}, {\bf50} (2022), 76–138.
	
\bibitem{Christopher}
G. S. Lakeland, C. J. Leininger, {\it Systoles and dehn surgery for hyperbolic 3–manifolds}, Algebr. Geom. Topol., {\bf14} (3) (2014), 1441--1460.

\bibitem{MG}
Gromov, Mikhael,{\it Filling Riemannian manifolds},{Journal of Differential Geometry}, {\bf18} (1983), no. 1, 1–147.

\bibitem{GP2}
Grisha Perelman, The entropy formula for the ricci flow and its geometric applications,Preprint 2002, 39 pages.

\bibitem{lackenby1998word}
M. Lackenby, {\it Word hyperbolic Dehn surgery}, preprint, \url{https://arxiv.org/abs/math/9808120}, (1998).

\bibitem{JP}
Purcell, Jessica S.,{\it Hyperbolic knot theory},{Graduate Studies in Mathematics},American Mathematical Society, Providence, RI, [2020], ©2020. xviii+369 pp. ISBN: 978-1-4704-5499-9

\bibitem{M}
R. Meyerhoff, {\it Sphere packing and volume in hyperbolic 3-space}, Comment. Math. Helv.{\bf61} (1986), 271–278
	
\bibitem{palaparthi2013closed}	 
S. Palaparthi, {\it Closed geodesic lengths in hyperbolic link complements in s3}, Int. J. Pure and Applied Math., {\bf83} (2013) 45–53.

\bibitem{Thu1}
William P. Thurston, {\it Three-dimensional manifolds,  Kleinian groups and hyperbolic geometry}, Bull. Amer. Math. Soc. (N.S.){\bf 6 }(1982), no. 3, 357–381.
\bibitem{Thu2}
William P. Thurston, {\it The geometry and topology of three manifolds }, Princeton Univ. Math. Dept. Notes, 1979.
\end{thebibliography}
\end{document}